\documentclass[11pt]{article}
\usepackage[utf8]{inputenc}
\usepackage{amsmath, amsthm, amssymb, mathrsfs}
\usepackage{geometry}
\usepackage{setspace}
\theoremstyle{plain}
\newtheorem{theorem}{Theorem}
\newtheorem{lemma}[theorem]{Lemma}
\newtheorem{proposition}[theorem]{Proposition}
\newtheorem{corollary}[theorem]{Corollary}

\theoremstyle{definition}
\newtheorem{definition}[theorem]{Definition}
\newtheorem{remark}[theorem]{Remark}
\newtheorem{question}[theorem]{Question}

\begin{document}

\title{\textbf{On $c_0$ Quotients of Nonseparable
Bourgain--Pisier $\mathscr{L}_\infty$-Spaces}}
\author{Kartik Patri\\Northeastern University\\ \texttt{patri.k@northeastern.edu}}
\date{February 2026}
\maketitle

\begin{abstract}
Lopez-Abad extended the Bourgain--Pisier construction to Banach spaces of
arbitrary density: every Banach space $X$ embeds isometrically into an
$\mathscr{L}_\infty$-space $Y$ such that $Y/X$ has the Radon--Nikod\'ym and
Schur properties. We study $c_0$ quotients of these spaces. For every $X$
of uncountable density $\kappa$, we describe a realization of the
construction that admits an explicit surjection $T: Y \to c_0$, given by
coordinate functionals built along the construction, whose kernel is an
$\mathscr{L}_\infty$-space of density $\kappa$. We also note that an
$\mathscr{L}_\infty$-space with a non-Grothendieck quotient has a $c_0$
quotient, and that every Lopez-Abad space has an $\ell_2$ quotient; all
of them have $c_0$ quotients if no infinite-dimensional Banach space is
both Schur and Grothendieck, a question we leave open. In an appendix we
characterize when the kernel of a quotient map from an
$\mathscr{L}_\infty$-space onto $c_0$ is an $\mathscr{L}_\infty$-space,
give a quotient map from $C[0,1]$ onto $c_0$ whose kernel is not, and show
that every non-Grothendieck $\mathscr{L}_\infty$-space admits one whose
kernel is. This version corrects errors in version~1.
\end{abstract}

\section{Introduction}

Mazur's separable quotient problem asks whether every infinite-dimensional
Banach space admits an infinite-dimensional separable quotient. Affirmative
answers are known for many classes, including reflexive spaces, weakly
compactly generated spaces and dual spaces; see \cite{MY20, FKLS18} for
surveys. For $\mathscr{L}_\infty$-spaces the answer is positive in
ZFC by known results (Remark~\ref{rmk:referees}), so the interest in this
class lies in the structure of their separable quotients rather than in
their existence.

The most natural separable quotient of an $\mathscr{L}_\infty$-space is
$c_0$. By a theorem of Alspach and Benyamini \cite{AB79}, every separable
infinite-dimensional $\mathscr{L}_\infty$-space has a quotient isomorphic
to $c_0$, and Argyros, Gasparis and Motakis \cite{AM16} showed that the
quotient map can be chosen so that its kernel is again an
$\mathscr{L}_\infty$-space. Their proof goes through a representation of
every separable $\mathscr{L}_\infty$-space as a Bourgain--Delbaen space. In
the nonseparable setting, a nonseparable $\mathscr{L}_\infty$-space has a
$c_0$ quotient exactly when it is not Grothendieck
(Theorem~\ref{thm:linftygrothendieck}), but no analogue of the
Bourgain--Delbaen representation is available.

We study these questions for the nonseparable Bourgain--Pisier spaces
constructed by Lopez-Abad \cite{Lo13}, which we call \emph{Lopez-Abad
spaces}. Lopez-Abad showed that every Banach space $X$ embeds
isometrically into an $\mathscr{L}_\infty$-space $Y$ of the same density
such that $Y/X$ has the Radon--Nikod\'ym and Schur properties; this extends
the Bourgain--Pisier construction \cite{BP83} from separable to arbitrary
Banach spaces. The construction is an inductive limit of Kisliakov
pushouts indexed by the finite subsets of the density, and it involves free
choices at each stage. Our main result shows that a suitable choice makes
the $c_0$ quotient, and its kernel, completely explicit.

In this paper we prove the following.

\begin{theorem}[General Criterion]
\label{thm:general}
Let $Y$ be an $\mathscr{L}_\infty$-space containing a closed subspace $X$
such that $Y/X$ is infinite-dimensional and not Grothendieck. Then $Y$
admits a quotient isomorphic to $c_0$.
\end{theorem}

\begin{corollary}[Lopez-Abad Spaces]
\label{cor:lopezabad}
Let $X$ be an infinite-dimensional Banach space of density
$\kappa \geq \aleph_1$ and let $Y \supseteq X$ be any Lopez-Abad
$\mathscr{L}_\infty$-space over $X$ \cite[Theorem~3.1]{Lo13}. Assume that
$Y/X$ is infinite-dimensional; this holds whenever $X$ is not an
$\mathscr{L}_\infty$-space, and, for every $X$, in the realization of
Section~\ref{sec:coords_realization}. Then:
\begin{enumerate}
    \item[(a)] $Y$ admits a quotient isomorphic to $\ell_2$; in particular
    $Y$ has an infinite-dimensional separable quotient.
    \item[(b)] If no infinite-dimensional Banach space is simultaneously
    Schur and Grothendieck (see Question~\ref{q:schurgroth}), then $Y$
    admits a quotient isomorphic to $c_0$.
\end{enumerate}
\end{corollary}

\begin{theorem}[Explicit Construction]
\label{thm:explicit}
Let $X$ be an infinite-dimensional Banach space of density
$\kappa \geq \aleph_1$, and fix $\lambda > 1$, $0 < \eta < 1$ with
$\lambda\eta > 1$. Then the realization of the Lopez-Abad construction
described in Section~\ref{sec:coords_realization} produces an
$\mathscr{L}_\infty$-space $Y \supseteq X$ of density $\kappa$, with $Y/X$
Schur and Radon--Nikod\'ym, together with an explicit bounded linear
surjection $T: Y \to c_0$ such that:
\begin{enumerate}
    \item $\|T\| \leq 1$, $T(X) = \{0\}$, and every $a \in c_0$ has a
    preimage $y$ with $\|y\| \leq \frac{2}{1-\eta}\|a\|$;
    \item $\ker T$ is an $\mathscr{L}_{\infty,\mu}$-space for every
    $\mu > \lambda \max\bigl(1, \frac{\eta}{1-\eta}\bigr)$, and
    $\operatorname{dens}(\ker T) = \kappa$.
\end{enumerate}
\end{theorem}

\begin{remark}
Throughout this paper, a \emph{Lopez-Abad space over $X$} means a space
$Y$ produced by the method of \cite[Theorem~3.1]{Lo13}, which generalizes
the Bourgain--Pisier construction \cite{BP83} to Banach spaces of
arbitrary density. The construction involves free choices at each stage
(Section~\ref{sec:lopezabad}); a specific set of choices is called a
\emph{realization}. Theorem~\ref{thm:explicit} concerns one specific
realization, introduced in Section~\ref{sec:coords_realization}, which
exists for every $X$.
\end{remark}

\subsection{Overview}

Section~\ref{sec:lopezabad} recalls the construction of \cite{Lo13} in the
detail needed for the proofs, correcting a misprint in the parameter
condition (Remark~\ref{rmk:parameters}). Section~\ref{sec:general} proves
Theorem~\ref{thm:general} and Corollary~\ref{cor:lopezabad}, and records
the open Question~\ref{q:schurgroth}.

The proof of Theorem~\ref{thm:explicit} occupies
Sections~\ref{sec:coords_realization}--\ref{sec:kernel}. In
Section~\ref{sec:coords_realization} we choose the isomorphisms
$\upsilon_t$ of the construction so that each stage carries
\emph{reserved coordinates} $c_{t,\gamma}$, one for each $\gamma \in t$, on
which the identified subspace $S_t$ vanishes. This costs nothing and is
available for every $X$. In Section~\ref{sec:coords} we use the reserved
coordinates to define, by recursion along the anti-lexicographic order,
functionals $e^*_\gamma$ of norm at most one on $Y$, compatible with the
whole directed system; the main point is that their definition does not
depend on the local system in which it is carried out. The map
$T = (e^*_{\gamma_k})_k$ is then bounded into $c_0$
(Proposition~\ref{prop:Tbounded}) and onto
(Theorem~\ref{thm:surjectivity}), the latter by lifting finitely supported
vectors at a single stage. In Section~\ref{sec:kernel} the kernel
is described stage by stage as a graph over finitely many coordinates,
which shows that it is an $\mathscr{L}_\infty$-space.
Section~\ref{sec:necessity} explains why the reserved coordinates are
needed.

Appendix~\ref{app:kernels} treats kernels of quotient maps onto $c_0$ in
general: whether the kernel is an $\mathscr{L}_\infty$-space is equivalent
to a complementation property in the dual (Lemma~\ref{lem:kernelcriterion}),
this can fail for some quotient maps from $C[0,1]$
(Proposition~\ref{prop:badkernel}, using Bourgain's badly complemented
$\ell_1$-subspaces of $L_1$ \cite{Bo81}), and every non-Grothendieck
$\mathscr{L}_\infty$-space has at least one quotient map onto $c_0$ with
$\mathscr{L}_\infty$ kernel (Theorem~\ref{thm:goodkernel}, via the
Kadec--Pe\l czy\'nski method \cite{KP62}). We make no claim of novelty for
the results of the appendix.

\subsection{Changes from Version~1}

Version~1 of this paper contained errors, several of them pointed out by
referees, to whom we are grateful. The main changes are as follows.
\begin{enumerate}
    \item Version~1 claimed to resolve Mazur's problem for Lopez-Abad
    spaces. As observed by the referees, the existence of separable
    quotients for these spaces (indeed, for all $\mathscr{L}_\infty$-spaces)
    follows from known results; see Remark~\ref{rmk:referees}.
    \item The general criterion of version~1 assumed that $Y/X$ is
    infinite-dimensional and Schur, and relied on the assertion that no
    infinite-dimensional Banach space is both Schur and Grothendieck. The
    proof of that assertion was incorrect, and we do not know whether it
    holds (Question~\ref{q:schurgroth}). Theorem~\ref{thm:general} now
    assumes directly that $Y/X$ is not Grothendieck, and the corresponding
    part of Corollary~\ref{cor:lopezabad} is conditional.
    \item The explicit construction of version~1 rested on a ``coordinate
    embedding'' assumption whose availability was not justified, and its proof
    of surjectivity contained a gap. Both are replaced by the
    reserved-coordinate realization and the new arguments of
    Sections~\ref{sec:coords_realization}--\ref{sec:kernel}, so
    Theorem~\ref{thm:explicit} now holds for every $X$.
    \item The parameter condition is corrected to $\lambda\eta > 1$, the
    $\mathscr{L}_\infty$-constants are stated correctly, and the
    ``necessity'' results of version~1 are reduced to the remarks of
    Section~\ref{sec:necessity}.
    \item Appendix~\ref{app:kernels} is new.
\end{enumerate}

\section{The Lopez-Abad Construction}
\label{sec:lopezabad}

We recall the construction from \cite{Lo13} in the detail needed below.
The goal is to embed a given Banach space $X$ isometrically into an
$\mathscr{L}_\infty$-space $Y$ so that $Y/X$ has the Radon--Nikod\'ym and
Schur properties. The construction generalizes the Bourgain--Pisier method
\cite{BP83} by replacing the linear inductive system of the separable case
with a directed system indexed by the finite subsets of the density of $X$.

\subsection{The Kisliakov Pushout}

The local step is due to Kisliakov \cite{Ki76}; we follow \cite{Lo13}.
Given Banach spaces $S \subseteq B$ and $E$, and an operator
$u: S \to E$ with $\|u\| \leq \eta \leq 1$, form
\begin{equation}
    (B \oplus_1 E)/N_u, \qquad N_u = \{(s, -u(s)) : s \in S\},
\end{equation}
where $\|(b,e)\| = \|b\| + \|e\|$. The maps
$i_E(e) = [(0,e)]$ and $i_B(b) = [(b,0)]$ are, respectively, an isometric
embedding and a contraction: for $e \in E$ and $\xi \in S$,
$\|\xi\| + \|e + u(\xi)\| \geq (1-\eta)\|\xi\| + \|e\| \geq \|e\|$.

An isometric embedding $j: E \to E_1$ is \emph{$\eta$-admissible} if it
arises from such a diagram \cite[Definition~2.3, Proposition~2.4]{Lo13}.
The key metric property \cite[(6)]{Lo13} is: if
$x_0, \ldots, x_n \in E_1$ and $x_0 + \cdots + x_n \in j(E)$, then
\begin{equation}
    \sum_{i=0}^n \|x_i\| \geq \Bigl\|\sum_{i=0}^n x_i\Bigr\|
    + (1-\eta)\sum_{i=0}^n \|q(x_i)\|,
\end{equation}
where $q: E_1 \to E_1/j(E)$ is the quotient map. Applied to the quotients
$E_1/j(E)$ arising along the construction, this inequality is what yields
the Radon--Nikod\'ym and Schur properties of $Y/X$ \cite[Theorem~2.5]{Lo13}.

\subsection{Parameters}

Fix an infinite-dimensional Banach space $X$ of density
$\kappa \geq \aleph_1$, a dense set $D = \{d_\alpha : \alpha < \kappa\}$
in $X$, and constants
\[
\lambda > 1, \qquad 0 < \eta < 1, \qquad \lambda\eta > 1.
\]
For $s \in [\kappa]^{<\omega}$ let
$X_s = \operatorname{span}\{d_\alpha : \alpha \in s\}$.

\begin{remark}[On the parameter condition]
\label{rmk:parameters}
In \cite{Lo13} (in the paragraph preceding Lemma~3.4, in both the arXiv
and the published versions) the condition is stated as
$\lambda\eta < 1$, while \cite[proof of Lemma~3.5, (b)]{Lo13} requires isomorphisms
$\upsilon_t$ with $\|\upsilon_t\| \leq \eta$ and
$\|\upsilon_t^{-1}\| \leq \lambda$. Since
$1 \leq \|\upsilon_t\|\|\upsilon_t^{-1}\| \leq \lambda\eta$, these are
incompatible, and the intended condition is $\lambda\eta > 1$. This is
exactly what the construction uses: $\eta < 1$ is needed for the Schur and
Radon--Nikod\'ym properties \cite[Theorem~2.5]{Lo13};
$\|\upsilon_t^{-1}\| \leq \lambda$ is needed for
Lemma~\ref{lem:Glinf}; and $\lambda\eta > 1$ is what makes such
$\upsilon_t$ exist (Lemma~\ref{lem:realization}). Indeed, the
construction in \cite[proof of Lemma~3.5]{Lo13} itself chooses
$T: Y \to \ell_\infty^{n}$ with $(\eta\lambda)^{-1}\|x\| \leq \|Tx\| \leq
\|x\|$ and sets $\upsilon = \lambda^{-1}T^{-1}$, which requires
$\eta\lambda \geq 1$.
\end{remark}

\subsection{The Anti-Lexicographic Order and Local Systems}

The \emph{anti-lexicographic order} $\prec$ on $[\kappa]^{<\omega}$
\cite[Definition~3.3]{Lo13} is given by $\emptyset \prec t$ for every
nonempty $t$ and, for nonempty $t,u$,
\[
t \prec u \iff \max t < \max u, \text{ or } \max t = \max u
\text{ and } t \setminus \{\max t\} \prec u \setminus \{\max u\}.
\]
It is a well-ordering extending strict inclusion: if $t \subsetneq s$ then
$t \prec s$ (induct on $|s|$, comparing maxima). For finite $s$ and
nonempty $t \subseteq s$, let $\bar{t}(s)$ denote the immediate
$\prec$-predecessor of $t$ in the power set $\mathcal{P}(s)$, and write
$\bar{t} := \bar{t}(t)$. Note that $\bar t \subseteq t$, and that
\begin{equation}\label{eq:predorder}
    \bar t \preceq \bar t(r) \preceq \bar t(s)
    \qquad \text{whenever } t \subseteq r \subseteq s,
\end{equation}
since $\bar t(s)$ is the $\prec$-maximum of the family
$\{w \subseteq s : w \prec t\}$, which grows with $s$.

The construction \cite[proof of Lemma~3.5]{Lo13} proceeds by
$\prec$-recursion on nonempty finite $t$. At stage $t$ one chooses,
\emph{once and for all} (independently of any $s \supseteq t$):
\begin{itemize}
    \item an integer $n_t$, a subspace $S_t \subseteq \ell_\infty^{n_t}$,
    and an isomorphism $\upsilon_t : S_t \to E^{(t)}_{\bar t}$ with
    \begin{equation}\label{eq:upsilon}
        \|\upsilon_t\| \leq \eta, \qquad \|\upsilon_t^{-1}\| \leq \lambda,
        \qquad \upsilon_t(S_t) = Y_t
    \end{equation}
    \cite[proof of Lemma~3.5, (b) and (16)]{Lo13}, where
    \begin{equation}\label{eq:Yt}
        Y_t := \operatorname{span}\Bigl(\bigcup_{w \preceq \bar t,\ w
        \subseteq t} j^{(t)}_{w,\bar t}\bigl(G^{(t)}_w\bigr) \cup
        j^{(t)}_{\emptyset,\bar t}(X_t)\Bigr) \subseteq E^{(t)}_{\bar t}
    \end{equation}
    (with $j^{(t)}_{\bar t,\bar t} = \mathrm{Id}$).
\end{itemize}
Then, for every finite $s \supseteq t$, the local spaces are defined by
$E^{(s)}_\emptyset = X$ and
\begin{equation}\label{eq:pushout}
    E_t^{(s)} = \bigl(\ell_\infty^{n_t} \oplus_1
    E_{\bar t(s)}^{(s)}\bigr)/N_t^{(s)}, \qquad
    N_t^{(s)} = \{(\xi, -v_t^{(s)}\xi) : \xi \in S_t\},
\end{equation}
where
\begin{equation}\label{eq:vts}
    v_t^{(s)} := j^{(s)}_{\bar t, \bar t(s)} \circ k^{(t,s)}_{\bar t} \circ
    \upsilon_t : S_t \to E^{(s)}_{\bar t(s)},
\end{equation}
with one-step embedding $j^{(s)}_{\bar t(s),t}(e) = [(0,e)]$ and
distinguished subspace
\begin{equation}
    G_t^{(s)} = \{[(x, 0)] : x \in \ell_\infty^{n_t}\} \subseteq E_t^{(s)}.
\end{equation}
Since the maps $j$ and $k$ are isometric embeddings,
$\|v^{(s)}_t\| \leq \eta$ and $\|v^{(s)}_t \xi\| \geq \lambda^{-1}\|\xi\|$
for $\xi \in S_t$.

We use the following properties from \cite[Lemma~3.5 and its proof]{Lo13}:
\begin{enumerate}
    \item[(P1)] (Composition, \cite[(18)]{Lo13}) For $u \prec t$ in
    $\mathcal{P}(s)$:
    $j^{(s)}_{u,t} = j^{(s)}_{\bar t(s),t} \circ j^{(s)}_{u,\bar t(s)}$.
    \item[(P2)] (Transition maps, \cite[(19)]{Lo13}) For
    $u \subseteq t \subseteq s$ the maps
    $k_u^{(t,s)}: E_u^{(t)} \to E_u^{(s)}$ are $\eta$-admissible isometric
    embeddings given by
    \begin{equation}\label{eq:transition}
        k_u^{(t,s)}\bigl([(x,y)]\bigr) =
        \bigl[\bigl(x,\; j^{(s)}_{\bar u(t), \bar u(s)} \circ
        k^{(t,s)}_{\bar u(t)}(y)\bigr)\bigr],
        \qquad x \in \ell_\infty^{n_u},\ y \in E^{(t)}_{\bar u(t)};
    \end{equation}
    in particular they preserve the $\ell_\infty^{n_u}$-component exactly.
    Moreover $k^{(s,s)}_u = \mathrm{Id}$ and
    $k^{(t,s)}_\emptyset = \mathrm{Id}_X$ \cite[proof of Lemma~3.5, (a) and the case $t = s$]{Lo13}.
    \item[(P3)] (Coherence, \cite[(11)]{Lo13}) For $v, u \subseteq t \subseteq s$
    with $v \prec u$: $k_u^{(t,s)} \circ j^{(t)}_{v,u} = j^{(s)}_{v,u} \circ
    k^{(t,s)}_v$.
    \item[(P4)] \cite[(13)]{Lo13} $k_u^{(t,s)}(G_u^{(t)}) = G_u^{(s)}$.
\end{enumerate}

\begin{lemma}
\label{lem:Glinf}
For every $x \in \ell_\infty^{n_t}$:
$\lambda^{-1}\|x\|_\infty \leq \|[(x,0)]\|_{E_t^{(s)}} \leq \|x\|_\infty$.
In particular $x \mapsto [(x,0)]$ is a $\lambda$-isomorphism of
$\ell_\infty^{n_t}$ onto $G_t^{(s)}$.
\end{lemma}

\begin{proof}
The upper bound holds because $\|(x,0)\| = \|x\|$. For the lower bound,
$\|[(x,0)]\| = \inf_{\xi \in S_t}\bigl(\|x - \xi\| + \|v^{(s)}_t \xi\|\bigr)$,
and for each $\xi \in S_t$,
\[
\|x\| \leq \|x - \xi\| + \|\xi\| \leq \|x - \xi\| + \lambda\|v_t^{(s)}\xi\|
\leq \lambda\bigl(\|x - \xi\| + \|v_t^{(s)}\xi\|\bigr),
\]
using $\lambda \geq 1$.
\end{proof}

\subsection{The Global System}

The global directed system is $E_s := E_s^{(s)}$, $G_s := G_s^{(s)}$, and
\begin{equation}
    j_{t,s} := j_{t,s}^{(s)} \circ k_t^{(t,s)}: E_t \to E_s
    \qquad (t \subseteq s)
\end{equation}
\cite[(15), Claim~4]{Lo13}. Moreover $j_{t,s}(G_t) \subseteq G_s$
\cite[Claim~5(b)]{Lo13}. Write $j_{s,\infty}: E_s \to E$ for the
canonical isometric embeddings into the completion $E$ of the inductive
limit, so that $j_{s,\infty} = j_{t,\infty} \circ j_{s,t}$ for
$s \subseteq t$. The space $Y$ is the closure of
$\bigcup_s j_{s,\infty}(G_s)$ in $E$, and $X$ embeds isometrically into
$Y$ via the maps $j_{\emptyset,s}|_{X_s}$ \cite[proof of Theorem~3.1]{Lo13}.

\begin{lemma}[Directed unions]
\label{lem:directed}
Let $Z$ be a Banach space and $(H_i)_{i \in I}$ an upward-directed family
of finite-dimensional subspaces with dense union, each satisfying
$d(H_i, \ell_\infty^{\dim H_i}) \leq C$. Then $Z$ is an
$\mathscr{L}_{\infty,\mu}$-space for every $\mu > C$.
\end{lemma}

\begin{proof}
Let $F \subseteq Z$ be finite-dimensional, $d = \dim F$, and $\theta > 0$.
Take an Auerbach basis $(f_l)_{l \leq d}$ of $F$ and Hahn--Banach
extensions $f_l^* \in Z^*$ of its coordinate functionals, so
$\|f_l\| = \|f_l^*\| = 1$ and $f^*_l(f_k) = \delta_{lk}$. By density and
directedness there are $i \in I$ and $f'_l \in H_i$ with
$\sum_l \|f'_l - f_l\| < \theta$. Define
$U = I_Z + \sum_l f^*_l(\cdot)\,(f'_l - f_l)$. Then
$\|U - I_Z\| < \theta$, so $U$ is invertible with $\|U\| \leq 1 + \theta$
and $\|U^{-1}\| \leq (1-\theta)^{-1}$, and $Uf_l = f'_l$. Hence
$H := U^{-1}(H_i)$ is a finite-dimensional subspace of $Z$ containing
$U^{-1}(\operatorname{span}\{f'_l\}) = F$, and
$d(H, \ell_\infty^{\dim H}) \leq C(1+\theta)/(1-\theta)$. Since $\theta$
is arbitrary, the claim follows.
\end{proof}

\begin{theorem}[\cite{Lo13}, Theorem~3.1]
\label{thm:lopezabad}
The space $Y$ has density $\kappa$, contains $X$ isometrically, $Y/X$ has
the Radon--Nikod\'ym and Schur properties, and $Y$ is an
$\mathscr{L}_{\infty,\mu}$-space for every $\mu > \lambda$.
\end{theorem}

The last statement follows from Lemmas~\ref{lem:Glinf}
and~\ref{lem:directed}, since $(j_{s,\infty}(G_s))_s$ is upward-directed
with dense union in $Y$.

\section{General Existence Theorem}
\label{sec:general}

\subsection{Background}

A Banach space $Z$ is \emph{Grothendieck} if every weak$^*$ convergent
sequence in $Z^*$ is weakly convergent. It has the \emph{Schur property}
if every weakly convergent sequence in $Z$ is norm convergent.

\begin{theorem}
\label{thm:linftygrothendieck}
An $\mathscr{L}_\infty$-space is Grothendieck if and only if it has no
quotient isomorphic to $c_0$.
\end{theorem}

\begin{proof}
By R\"abiger's theorem \cite[Theorem~3.1.7]{GK21}, a Banach space $Z$ is
Grothendieck if and only if $Z^*$ is weakly sequentially complete and $Z$
has no quotient isomorphic to $c_0$. So it suffices to show that $Z^*$ is
weakly sequentially complete when $Z$ is an $\mathscr{L}_\infty$-space.
By \cite[Proposition~4.3.3]{GK21}, $Z^{**}$ is injective, hence
complemented in $\ell_\infty(\Gamma) = C(\beta\Gamma)$ for a suitable set
$\Gamma$. Taking adjoints, $Z^{***}$ is complemented in
$C(\beta\Gamma)^* = M(\beta\Gamma)$, which is an $L_1(\mu)$-space, and
$Z^*$ is complemented in $Z^{***}$. Thus $Z^*$ is isomorphic to a
complemented subspace of an $L_1(\mu)$-space, and is therefore weakly
sequentially complete.
\end{proof}

\begin{remark}
Version~1 attributed this statement to \cite[Proposition~3.1.8]{GK21}.
That proposition is stated for spaces with local unconditional structure;
it does cover $\mathscr{L}_\infty$-spaces (their biduals are injective,
hence complemented in a Banach lattice), but the argument above, via
R\"abiger's theorem, is more direct.
\end{remark}

\begin{lemma}
\label{lem:quotientgroth}
Quotients of Grothendieck spaces are Grothendieck.
\end{lemma}

\begin{proof}
Let $Q: Y \to Y/X$ be the quotient map, $Y$ Grothendieck, and
$f_k \xrightarrow{w^*} 0$ in $(Y/X)^*$. Then $Q^*f_k(y) = f_k(Qy) \to 0$
for every $y$, so $Q^*f_k \xrightarrow{w^*} 0$ in $Y^*$, hence weakly. Since
$Q^*$ is an isometric embedding, $Q^{**}: Y^{**} \to (Y/X)^{**}$ is onto;
given $\Phi \in (Y/X)^{**}$ choose $\Psi$ with $Q^{**}\Psi = \Phi$, and
then $\Phi(f_k) = \Psi(Q^*f_k) \to 0$.
\end{proof}

\subsection{Proofs}

\begin{proof}[Proof of Theorem~\ref{thm:general}]
By Lemma~\ref{lem:quotientgroth}, $Y$ is not Grothendieck, since its
quotient $Y/X$ is not. As $Y$ is an $\mathscr{L}_\infty$-space,
Theorem~\ref{thm:linftygrothendieck} gives a quotient isomorphic to $c_0$.
\end{proof}

\begin{proof}[Proof of Corollary~\ref{cor:lopezabad}]
First, the assumption. If $Y/X$ were finite-dimensional, $X$ would be a
complemented subspace of finite codimension in the $\mathscr{L}_\infty$-space
$Y$, hence itself an $\mathscr{L}_\infty$-space \cite{LR69}. For the
realization of Section~\ref{sec:coords_realization},
Proposition~\ref{prop:Tbounded} shows that $T(X) = 0$, so $T$ induces a
surjection of $Y/X$ onto $c_0$, and $Y/X$ is infinite-dimensional.

(a) By Theorem~\ref{thm:lopezabad}, $Z := Y/X$ is infinite-dimensional
with the Schur property. Choose a bounded sequence in $Z$ with no
norm-convergent subsequence. By Rosenthal's $\ell_1$ theorem \cite{Ro74}
it has a subsequence that is either weakly Cauchy or equivalent to the
unit vector basis of $\ell_1$. In a Schur space weakly Cauchy sequences are
norm Cauchy, so the first alternative is excluded, and $Z$ contains a
copy of $\ell_1$. By Mujica's theorem \cite{Mu97}, $Z$ has a quotient
isomorphic to $\ell_2$, and composing with $Y \to Y/X$ gives one for $Y$.

(b) Under the stated hypothesis, the infinite-dimensional Schur space
$Y/X$ is not Grothendieck, and Theorem~\ref{thm:general} applies.
\end{proof}

\begin{remark}
\label{rmk:referees}
The argument in (a) was pointed out by a referee. Another referee observed
that \emph{every} $\mathscr{L}_\infty$-space has an infinite-dimensional
separable quotient: if it is not Grothendieck, use
Theorem~\ref{thm:linftygrothendieck}; if it is Grothendieck and its dual
ball is not weak$^*$ sequentially compact, use
\cite[Theorem~18]{FKLS18}; otherwise the space is reflexive. Thus the
separable-quotient conclusion for Lopez-Abad spaces is not new; the point
of this paper is the explicit $c_0$ quotient of
Theorem~\ref{thm:explicit}.
\end{remark}

\begin{question}
\label{q:schurgroth}
Can an infinite-dimensional Banach space have both the Schur property and
the Grothendieck property?
\end{question}

\begin{remark}
Version~1 of this paper asserted a negative answer to
Question~\ref{q:schurgroth}, but its proof relied on the false claim that
weakly compact, completely continuous operators are compact (the formal
identity $\ell_1 \to \ell_2$ is a counterexample). We do not know a proof
or a reference. A negative answer would give
Corollary~\ref{cor:lopezabad}(b) for every realization; for the
realization of Section~\ref{sec:coords_realization},
Theorem~\ref{thm:explicit} gives the $c_0$ quotient unconditionally.

Question~\ref{q:schurgroth} is related to the open problem of whether
every Grothendieck space has Pe\l czy\'nski's property~(V)
\cite[Problem~4]{GK21}. Indeed, a Schur space contains no copy of $c_0$,
so every operator from it is unconditionally converging; if such a space
had property~(V), its identity would be weakly compact, so the space would
be reflexive, and a reflexive Schur space is finite-dimensional. Hence a
positive answer to \cite[Problem~4]{GK21} would answer
Question~\ref{q:schurgroth} negatively. In particular, by
\cite[Proposition~3.2.7]{GK21}, no infinite-dimensional Schur space that is
isomorphic to a complemented subspace of a Banach lattice is Grothendieck.
\end{remark}

\section{The Reserved-Coordinate Realization}
\label{sec:coords_realization}

We now fix a specific realization of the construction, by specifying the
choices of $(n_t, S_t, \upsilon_t)$ in~\eqref{eq:upsilon}.

\begin{lemma}
\label{lem:realization}
Fix $\delta > 0$ with $1 + \delta \leq \lambda\eta$. For every nonempty
finite $t$ there are an integer $m'_t$ and a linear map
$\varphi_t : Y_t \to \ell_\infty^{m'_t}$ with
$(1+\delta)^{-1}\|y\| \leq \|\varphi_t y\|_\infty \leq \|y\|$. Put
\[
n_t := m'_t + |t|,
\]
write $\ell_\infty^{n_t} = \ell_\infty^{m'_t} \oplus_\infty \ell_\infty^{t}$,
call the first $m'_t$ coordinates the \emph{block coordinates}
$B_t$ and denote the remaining coordinates by $c_{t,\gamma}$,
$\gamma \in t$ (the \emph{reserved coordinates}), and set
\[
S_t := \{(\varphi_t y, 0) : y \in Y_t\}, \qquad
\upsilon_t(\varphi_t y, 0) := \tfrac{\eta}{1+\delta}\, y .
\]
Then $\upsilon_t$ satisfies~\eqref{eq:upsilon}, and every $\xi \in S_t$
vanishes on all reserved coordinates.
\end{lemma}

\begin{proof}
For $\varphi_t$: let $\theta > 0$ with $(1-\theta)^{-1} \leq 1 + \delta$,
let $y_1, \ldots, y_{m}$ be a $\theta$-net of the unit sphere of the
finite-dimensional space $Y_t$, choose norm-one $g_i \in Y_t^*$ with
$g_i(y_i) = 1$, and set $\varphi_t y = (g_i(y))_{i \leq m}$, $m'_t = m$.
Then $\|\varphi_t\| \leq 1$, and for a unit vector $y$ with
$\|y - y_i\| \leq \theta$ we get $|g_i(y)| \geq 1 - \theta$.

For $\upsilon_t$: it maps $S_t$ onto $Y_t$, and
$\|\upsilon_t(\varphi_t y, 0)\| = \frac{\eta}{1+\delta}\|y\| \leq
\eta\|\varphi_t y\|$, while
$\|(\varphi_t y, 0)\| \leq \|y\| = \frac{1+\delta}{\eta}
\|\upsilon_t(\varphi_t y,0)\| \leq \lambda\|\upsilon_t(\varphi_t y,0)\|$.
The last assertion holds by the definition of $S_t$.
\end{proof}

The only requirements on $(n_t, S_t, \upsilon_t)$ in \cite{Lo13} are those
in~\eqref{eq:upsilon}, so Lemma~\ref{lem:realization} defines a valid
realization for every $X$. We call it the \emph{reserved-coordinate
realization} and use it for the rest of the paper.

\begin{remark}
Version~1 instead imposed a ``coordinate embedding'' assumption, that
$S_t$ be a coordinate subspace of $\ell_\infty^{n_t}$. That assumption
would force $S_t$ to be isometric to $\ell_\infty^{m_t}$, and hence $Y_t$
to be uniformly close to $\ell_\infty^{\dim Y_t}$, which is not
justified in general; the argument given for its availability (via
Auerbach bases) does not respect the norm bounds in~\eqref{eq:upsilon}.
The proofs only ever use the weaker property that $S_t$ vanishes on some
coordinates, which Lemma~\ref{lem:realization} provides for free.
\end{remark}

\section{Coordinate Functionals}
\label{sec:coords}

Fix $\gamma \in \kappa$. We construct a functional $e^*_\gamma$ on $Y$ by
defining, simultaneously for all finite $s$, functionals
$f_u^{(s)} = f_u^{(s),\gamma}$ on the local spaces $E_u^{(s)}$,
$u \subseteq s$, by $\prec$-recursion on $u$.

\begin{definition}
\label{def:functionals}
Set $f_\emptyset^{(s)} := 0$ on $E^{(s)}_\emptyset = X$. For nonempty $u$,
assuming the functionals $f_w^{(s')}$ have been defined for all
$w \prec u$ and all $s' \supseteq w$, define one functional
$\phi_u = \phi_u^{\gamma} \in (\ell_\infty^{n_u})^*$, the same for every
$s \supseteq u$:
\begin{itemize}
    \item if $\gamma \notin u$: $\phi_u := 0$;
    \item if $\gamma \in u$: $\phi_u := h_u + (1-\eta)\, e^*_{c_{u,\gamma}}$,
    where $e^*_{c_{u,\gamma}}$ is the coordinate functional of the reserved
    coordinate $c_{u,\gamma}$, and $h_u$ is supported on the block
    coordinates $B_u$ and satisfies
    \[
    h_u|_{S_u} = f^{(u)}_{\bar u} \circ \upsilon_u, \qquad
    \|h_u\|_{\ell_1^{n_u}} \leq \|f^{(u)}_{\bar u} \circ \upsilon_u\|
    \]
    (a Hahn--Banach extension from $S_u \subseteq \ell_\infty^{m'_u}
    \oplus 0$ to $\ell_\infty^{m'_u}$, extended by $0$ on the reserved
    coordinates; we take $h_u = 0$ whenever
    $f^{(u)}_{\bar u}\circ\upsilon_u = 0$).
\end{itemize}
Then, for every finite $s \supseteq u$, define
\[
f_u^{(s)}\bigl([(b,e)]\bigr) := \phi_u(b) + f^{(s)}_{\bar u(s)}(e),
\qquad b \in \ell_\infty^{n_u},\ e \in E^{(s)}_{\bar u(s)}.
\]
\end{definition}

\begin{proposition}
\label{prop:functionals}
For all finite $u \subseteq t \subseteq s$:
\begin{enumerate}
    \item \emph{(Well-defined)} $f_u^{(s)}$ is well defined on the quotient
    $E_u^{(s)}$.
    \item \emph{(Norm)} $\|f_u^{(s)}\| \leq 1$.
    \item \emph{(Local compatibility)} $f_u^{(s)} \circ j^{(s)}_{w,u} =
    f_w^{(s)}$ for every $w \prec u$ in $\mathcal{P}(s)$.
    \item \emph{(Coherence)} $f_u^{(s)} \circ k_u^{(t,s)} = f_u^{(t)}$.
    \item \emph{(Vanishing)} If $\gamma \notin u$, then $f_u^{(s)}$
    vanishes on $G_u^{(s)}$.
\end{enumerate}
\end{proposition}

\begin{proof}
By $\prec$-induction on $u$; the case $u = \emptyset$ is trivial. Let
$u \neq \emptyset$ and assume (1)--(5) for all $w \prec u$.

\emph{Step 0: an identity.} For $\xi \in S_u$ and $s \supseteq u$,
\begin{equation}\label{eq:keyidentity}
    f^{(s)}_{\bar u(s)}\bigl(v^{(s)}_u \xi\bigr) =
    f^{(s)}_{\bar u(s)}\bigl(j^{(s)}_{\bar u,\bar u(s)}
    k^{(u,s)}_{\bar u}\upsilon_u \xi\bigr) =
    f^{(s)}_{\bar u}\bigl(k^{(u,s)}_{\bar u}\upsilon_u\xi\bigr) =
    f^{(u)}_{\bar u}(\upsilon_u\xi),
\end{equation}
by~\eqref{eq:vts}; then (3) in $\mathcal{P}(s)$ for $\bar u(s)$ with
$w = \bar u$, which applies because $\bar u \preceq \bar u(s)$
by~\eqref{eq:predorder} (if they are equal there is nothing to do); then
(4) for the triple $\bar u \subseteq u \subseteq s$. All three are
available since $\bar u, \bar u(s) \prec u$.

\emph{Step 1: $\phi_u|_{S_u} = f^{(u)}_{\bar u} \circ \upsilon_u$.} If
$\gamma \in u$, this holds because the reserved coordinate $c_{u,\gamma}$
vanishes on $S_u$ (Lemma~\ref{lem:realization}) and by the choice of
$h_u$. (When $u = \{\gamma\}$ we have $\bar u = \emptyset$, so
$f^{(u)}_{\bar u}\circ\upsilon_u = 0$ and $h_u = 0$.) If $\gamma \notin u$,
then $\phi_u = 0$, and we must show that $f^{(u)}_{\bar u}$ vanishes on
$Y_u$ (see~\eqref{eq:Yt}). If $\bar u = \emptyset$, this is
$f_\emptyset = 0$. Otherwise, each $G$-term of $Y_u$ has the form
$j^{(u)}_{w,\bar u}(G^{(u)}_w)$ with $w \subseteq u$ and $w \preceq \bar u$;
since $\gamma \notin u$, also $\gamma \notin w$, so $f^{(u)}_w$ vanishes on
$G^{(u)}_w$ by (5) for $w \prec u$, and
$f^{(u)}_{\bar u} \circ j^{(u)}_{w,\bar u} = f^{(u)}_w$ by (3) for
$\bar u$ (trivially if $w = \bar u$). Likewise
$f^{(u)}_{\bar u} \circ j^{(u)}_{\emptyset,\bar u} = f_\emptyset = 0$ by
(3) for $\bar u$.

\emph{(1).} By~\eqref{eq:keyidentity} and Step~1, for $\xi \in S_u$,
$\phi_u(\xi) - f^{(s)}_{\bar u(s)}(v_u^{(s)}\xi) = 0$, so $f_u^{(s)}$
vanishes on $N^{(s)}_u$.

\emph{(2).} Since $\|f^{(u)}_{\bar u}\| \leq 1$ by (2) for $\bar u$ and
$\|\upsilon_u\| \leq \eta$, the functional $f^{(u)}_{\bar u}\circ
\upsilon_u$ has norm at most $\eta$ on $S_u$. Identifying $S_u$ with
$\varphi_u(Y_u) \subseteq \ell_\infty^{m'_u}$, its Hahn--Banach extension
to $\ell_\infty^{m'_u}$ has the same norm, and extending by $0$ on the
reserved coordinates does not change the $\ell_1$-norm; so
$\|h_u\|_1 \leq \eta$. As $h_u$ and $e^*_{c_{u,\gamma}}$ have disjoint
supports,
$\|\phi_u\|_{\ell_1^{n_u}} \leq \eta + (1-\eta) = 1$. The dual norm of
$(\ell_\infty^{n_u} \oplus_1 E)^*$ is the maximum of the two component
norms, so for any representative and any $\xi \in S_u$,
\[
|f^{(s)}_u([(b,e)])| = |\phi_u(b+\xi) + f^{(s)}_{\bar u(s)}(e -
v_u^{(s)}\xi)| \leq \|b+\xi\| + \|e - v^{(s)}_u\xi\|.
\]
Taking the infimum over $\xi$ gives $\|f_u^{(s)}\| \leq 1$.

\emph{(3).} For the one-step embedding,
$f_u^{(s)}([(0,e)]) = f^{(s)}_{\bar u(s)}(e)$. Let $w \prec u$ in
$\mathcal{P}(s)$; then $w \preceq \bar u(s)$ by maximality of
$\bar u(s)$. If $w = \bar u(s)$ we are done. Otherwise, by (P1),
$f^{(s)}_u \circ j^{(s)}_{w,u} = f^{(s)}_u \circ j^{(s)}_{\bar u(s),u}
\circ j^{(s)}_{w,\bar u(s)} = f^{(s)}_{\bar u(s)} \circ
j^{(s)}_{w,\bar u(s)} = f^{(s)}_w$, using (3) for $\bar u(s) \prec u$.

\emph{(4).} If $t = s$ there is nothing to prove, since
$k^{(s,s)}_u = \mathrm{Id}$. Otherwise, for $x \in \ell_\infty^{n_u}$ and
$y \in E^{(t)}_{\bar u(t)}$, by~\eqref{eq:transition},
\begin{align*}
f_u^{(s)}\bigl(k^{(t,s)}_u[(x,y)]\bigr)
&= \phi_u(x) + f^{(s)}_{\bar u(s)}\bigl(j^{(s)}_{\bar u(t),\bar u(s)}
k^{(t,s)}_{\bar u(t)} y\bigr)
= \phi_u(x) + f^{(s)}_{\bar u(t)}\bigl(k^{(t,s)}_{\bar u(t)} y\bigr) \\
&= \phi_u(x) + f^{(t)}_{\bar u(t)}(y) = f^{(t)}_u([(x,y)]),
\end{align*}
using (3) in $\mathcal{P}(s)$ for $\bar u(s)$ with $w = \bar u(t)$
(valid since $\bar u(t) \preceq \bar u(s)$ by~\eqref{eq:predorder}), and
(4) for the triple $\bar u(t) \subseteq t \subseteq s$ (valid since
$\bar u(t) \prec u$). Here we use that $\phi_u$ does not depend on $s$.

\emph{(5).} If $\gamma \notin u$, then
$f_u^{(u)}([(x,0)]) = \phi_u(x) = 0$; by (P4) and (4) the same holds on
$G^{(s)}_u = k^{(u,s)}_u(G^{(u)}_u)$.
\end{proof}

\begin{corollary}[The global functional]
\label{cor:global}
There is a unique $e^*_\gamma \in E^*$ with
$e^*_\gamma(j_{s,\infty} y) = f_s^{(s)}(y)$ for every finite $s$ and
$y \in E_s$, and $\|e^*_\gamma\| \leq 1$. Moreover:
\begin{enumerate}
    \item[(a)] \emph{(Vanishing)} If $\gamma \notin s$, then $e^*_\gamma$
    vanishes on $j_{s,\infty}(G_s)$.
    \item[(b)] \emph{(Reserved coordinates)} Let $x \in \ell_\infty^{n_u}$
    vanish on the block coordinates $B_u$. Then
    \[
    e^*_\gamma\bigl(j_{u,\infty}[(x,0)]\bigr) =
    \begin{cases} (1-\eta)\, x_{c_{u,\gamma}}, & \gamma \in u,\\
    0, & \gamma \notin u.\end{cases}
    \]
    \item[(c)] $e^*_\gamma$ vanishes on $X \subseteq Y$.
\end{enumerate}
\end{corollary}

\begin{proof}
For $t \subseteq s$, by Proposition~\ref{prop:functionals}(3),(4),
$f_s^{(s)} \circ j_{t,s} = f^{(s)}_s \circ j^{(s)}_{t,s} \circ
k^{(t,s)}_t = f^{(s)}_t \circ k^{(t,s)}_t = f^{(t)}_t$ (for $t \subsetneq s$
we have $t \prec s$). Hence the formula is consistent on the inductive
limit: if $j_{s,\infty}y = j_{t,\infty}z$, then with $r = s \cup t$ we have
$j_{s,r}y = j_{t,r}z$, so $f_s^{(s)}(y) = f^{(r)}_r(j_{s,r}y) =
f^{(r)}_r(j_{t,r}z) = f^{(t)}_t(z)$. The functional has norm at most $1$
on a dense subspace and extends to $E$. (a) is
Proposition~\ref{prop:functionals}(5). (b): the value is
$f^{(u)}_u([(x,0)]) = \phi_u(x)$; if $\gamma \notin u$ this is $0$, and
otherwise $h_u(x) = 0$ since $h_u$ is supported on $B_u$. (c): $X$ is the
closure of the union of the sets $j_{s,\infty}(j_{\emptyset,s}(X_s))$, and
$f^{(s)}_s \circ j_{\emptyset,s} = f^{(s)}_s \circ j^{(s)}_{\emptyset,s}
= f_\emptyset = 0$ by Proposition~\ref{prop:functionals}(3) (here
$k^{(\emptyset,s)}_\emptyset = \mathrm{Id}_X$).
\end{proof}

\section{Construction of the Quotient Map}

Fix pairwise distinct $\gamma_1, \gamma_2, \ldots \in \kappa$, and write
$e^*_k := e^*_{\gamma_k}|_Y$.

\begin{proposition}[The Quotient Map]
\label{prop:Tbounded}
$T(y) := (e^*_1(y), e^*_2(y), \ldots)$ defines a bounded linear operator
$T: Y \to c_0$ with $\|T\| \leq 1$ and $T(X) = \{0\}$.
\end{proposition}

\begin{proof}
By Corollary~\ref{cor:global}, $\|e^*_k\| \leq 1$, so
$\|T(y)\|_\infty \leq \|y\|$, and $T(X) = 0$ by
Corollary~\ref{cor:global}(c). It remains to show $T(y) \in c_0$. Given
$y \in Y$ and $\varepsilon > 0$, choose finite $s$ and $g \in G_s$ with
$\|y - j_{s,\infty}g\| < \varepsilon$. For all $k$ with $\gamma_k \notin s$
(all but finitely many), Corollary~\ref{cor:global}(a) gives
$|e^*_k(y)| \leq \|y - j_{s,\infty}g\| < \varepsilon$.
\end{proof}

\section{Surjectivity}

\begin{lemma}[Finite lifts]
\label{lem:finitelift}
Let $a = (a_1, \ldots, a_K, 0, 0, \ldots)$ be finitely supported and
$\sigma = \{\gamma_1, \ldots, \gamma_K\}$. Put
\[
x := \frac{1}{1-\eta}\sum_{k=1}^K a_k\, e_{c_{\sigma,\gamma_k}} \in
\ell_\infty^{n_\sigma}, \qquad y := j_{\sigma,\infty}[(x,0)] \in Y.
\]
Then $T(y) = a$ and $\|y\| \leq \|a\|_\infty/(1-\eta)$.
\end{lemma}

\begin{proof}
By Lemma~\ref{lem:Glinf}, $\|y\| \leq \|x\|_\infty = \|a\|_\infty/(1-\eta)$,
since the coordinates $c_{\sigma,\gamma_k}$ are distinct. The vector $x$
vanishes on $B_\sigma$, so Corollary~\ref{cor:global}(b) gives
$e^*_k(y) = (1-\eta)\cdot a_k/(1-\eta) = a_k$ for $k \leq K$, and
$e^*_k(y) = 0$ for $k > K$ because $\gamma_k \notin \sigma$.
\end{proof}

\begin{theorem}[Surjectivity of $T$]
\label{thm:surjectivity}
$T$ is onto, and every $a \in c_0$ has a preimage of norm at most
$\frac{2}{1-\eta}\|a\|_\infty$.
\end{theorem}

\begin{proof}
Let $0 \neq a \in c_0$. Choose $0 = N_0 < N_1 < N_2 < \cdots$ with
$\sup_{n > N_m}|a_n| \leq 2^{-m}\|a\|_\infty$, and let
$b_m := a \cdot \mathbf{1}_{(N_{m-1}, N_m]}$. Then $\|b_1\| \leq \|a\|$
and $\|b_m\| \leq 2^{-(m-1)}\|a\|$ for $m \geq 2$, so
$\sum_m \|b_m\| \leq 2\|a\|$, and
$\|a - \sum_{m \leq M} b_m\|_\infty = \sup_{n > N_M}|a_n| \to 0$.
Lemma~\ref{lem:finitelift} gives $y_m$ with $Ty_m = b_m$ and
$\|y_m\| \leq \|b_m\|/(1-\eta)$. The series $y = \sum_m y_m$ converges
absolutely, $\|y\| \leq 2\|a\|/(1-\eta)$, and by continuity
$Ty = \sum_m b_m = a$.
\end{proof}

\section{Main Explicit Construction Theorem}
\label{sec:mainexplicit}

\begin{proof}[Proof of Theorem~\ref{thm:explicit}]
The space $Y$ and its properties are given by Theorem~\ref{thm:lopezabad}
applied to the realization of Lemma~\ref{lem:realization}. Item~(1) is
Proposition~\ref{prop:Tbounded} and Theorem~\ref{thm:surjectivity}.
Item~(2) is Propositions~\ref{prop:kernel_linf}
and~\ref{prop:kernel_density} below.
\end{proof}

\begin{remark}[Kernels of quotient maps onto $c_0$]
\label{rmk:generalkernel}
Theorem~\ref{thm:explicit} concerns one explicit map, and it is natural to
ask how special an $\mathscr{L}_\infty$ kernel is. In
Appendix~\ref{app:kernels} we record three facts:
\begin{enumerate}
    \item[(i)] for an $\mathscr{L}_\infty$-space $Y$ and a surjection
    $S: Y \to c_0$, the kernel is an $\mathscr{L}_\infty$-space if and only
    if $S^*(\ell_1)$ is complemented in $Y^*$
    (Lemma~\ref{lem:kernelcriterion});
    \item[(ii)] not every surjection works: there is a surjection
    $C[0,1] \to c_0$ whose kernel is not an $\mathscr{L}_\infty$-space
    (Proposition~\ref{prop:badkernel});
    \item[(iii)] every non-Grothendieck $\mathscr{L}_\infty$-space, of any
    density, admits \emph{some} surjection onto $c_0$ with
    $\mathscr{L}_\infty$ kernel (Theorem~\ref{thm:goodkernel}).
\end{enumerate}
In particular, (iii) gives a $c_0$ quotient with $\mathscr{L}_\infty$
kernel for \emph{every} realization of the Lopez-Abad construction that is
not Grothendieck. The interest of Theorem~\ref{thm:explicit} is that the
quotient map and its kernel are explicit in terms of the construction. We
do not know whether (ii) and (iii) are in the literature, and we make no
claim of novelty for them.
\end{remark}

\section{Structure of the Kernel}
\label{sec:kernel}

Let $K := \ker T$, and for finite $s$ let
\[
\Gamma_s := s \cap \{\gamma_k : k \in \mathbb{N}\}, \qquad
K_s := j_{s,\infty}(G_s) \cap K, \qquad
\mu_0 := \lambda\max\Bigl(1, \frac{\eta}{1-\eta}\Bigr).
\]

\begin{lemma}
\label{lem:Ks_structure}
For each finite $s$:
\begin{enumerate}
    \item $d\bigl(K_s, \ell_\infty^{n_s - |\Gamma_s|}\bigr) \leq \mu_0$.
    \item If $s \subseteq t$ then $K_s \subseteq K_t$.
    \item $\bigcup_s K_s$ is dense in $K$.
\end{enumerate}
\end{lemma}

\begin{proof}
(1) Every element of $j_{s,\infty}(G_s)$ has the form
$y = j_{s,\infty}[(x,0)]$ with $x \in \ell_\infty^{n_s}$, and $x$ is unique
by Lemma~\ref{lem:Glinf}. By the defining property of $e^*_\gamma$
(Corollary~\ref{cor:global}) and Definition~\ref{def:functionals},
$e^*_\gamma(y) = f^{(s)}_s([(x,0)]) = \phi_s^\gamma(x)$ for every
$\gamma \in s$; and $e^*_k(y) = 0$ whenever $\gamma_k \notin s$, by
Corollary~\ref{cor:global}(a). Since $\gamma_k \in s$ exactly when
$\gamma_k \in \Gamma_s$, and
$\phi^\gamma_s(x) = h^\gamma_s(x) + (1-\eta)x_{c_{s,\gamma}}$ for
$\gamma \in s$, we get that $y \in K$ if and only if $x \in C_s$, where
\[
C_s := \Bigl\{x \in \ell_\infty^{n_s} :
x_{c_{s,\gamma}} = -\frac{h_s^{\gamma}(x)}{1-\eta}
\text{ for all } \gamma \in \Gamma_s\Bigr\},
\]
and $h^\gamma_s(x)$ depends only on the block part $x|_{B_s}$, because
$h^\gamma_s$ is supported on $B_s$. Let
$R : C_s \to \ell_\infty^{n_s - |\Gamma_s|}$ delete the coordinates
$c_{s,\gamma}$, $\gamma \in \Gamma_s$. The block coordinates are among
those retained, so the deleted coordinates are determined by $Rx$; hence
$R$ is a linear bijection, and $\|R\| \leq 1$. Since
$\|h_s^\gamma\|_1 \leq \eta$ (proof of
Proposition~\ref{prop:functionals}(2)) and
$\|x|_{B_s}\|_\infty \leq \|Rx\|_\infty$,
\[
\|x\|_\infty = \max\Bigl(\|Rx\|_\infty,\ \max_{\gamma \in \Gamma_s}
\frac{|h_s^\gamma(x|_{B_s})|}{1-\eta}\Bigr)
\leq \max\Bigl(1, \frac{\eta}{1-\eta}\Bigr) \|Rx\|_\infty ,
\]
so $\|R^{-1}\| \leq \max(1, \eta/(1-\eta))$. Let
$\psi(x) = j_{s,\infty}[(x,0)]$, so that $K_s = \psi(C_s)$; by
Lemma~\ref{lem:Glinf} and since $j_{s,\infty}$ is an isometry,
$\|\psi\| \leq 1$ and $\|\psi^{-1}\| \leq \lambda$ on $\psi(C_s)$. Then
$\Phi := R \circ \psi^{-1} : K_s \to \ell_\infty^{n_s - |\Gamma_s|}$ is an
isomorphism with $\|\Phi\|\|\Phi^{-1}\| \leq \lambda \cdot
\max(1, \eta/(1-\eta)) = \mu_0$.

(2) $j_{s,\infty}(G_s) = j_{t,\infty}(j_{s,t}G_s) \subseteq
j_{t,\infty}(G_t)$ by \cite[Claim~5(b)]{Lo13}.

(3) Let $y \in K$ and $\varepsilon > 0$. Since
$\bigcup_s j_{s,\infty}(G_s)$ is dense in $Y$, choose finite $s$ and
$g \in j_{s,\infty}(G_s)$ with $\|y - g\| < \varepsilon$. For
$\gamma \in \Gamma_s$ put $w_\gamma := j_{s,\infty}[(e_{c_{s,\gamma}},0)]$
and
\[
g_0 := g - \sum_{\gamma \in \Gamma_s}
\frac{e^*_\gamma(g)}{1-\eta}\, w_\gamma .
\]
By Corollary~\ref{cor:global}(b), $e^*_{\gamma'}(w_\gamma) =
(1-\eta)\delta_{\gamma\gamma'}$ for $\gamma, \gamma' \in \Gamma_s$, so
$e^*_{\gamma'}(g_0) = 0$ for $\gamma' \in \Gamma_s$; and for
$\gamma_k \notin s$, both $g$ and $w_\gamma$ lie in $j_{s,\infty}(G_s)$, so
$e^*_k(g_0) = 0$ by Corollary~\ref{cor:global}(a). These two cases cover
every $k$, so $g_0 \in K \cap j_{s,\infty}(G_s) = K_s$. For
$\gamma \in \Gamma_s$ we have $e^*_\gamma(y) = 0$ (as $y \in K$) and
$\|e^*_\gamma\| \leq 1$, so $|e^*_\gamma(g)| = |e^*_\gamma(g - y)| <
\varepsilon$. Since the coordinates $c_{s,\gamma}$ are distinct,
Lemma~\ref{lem:Glinf} gives $\|\sum_\gamma c_\gamma w_\gamma\| \leq
\|\sum_\gamma c_\gamma e_{c_{s,\gamma}}\|_\infty = \max_\gamma
|c_\gamma|$. Hence $\|g - g_0\| \leq \varepsilon/(1-\eta)$ and
$\|y - g_0\| \leq \varepsilon\bigl(1 + (1-\eta)^{-1}\bigr)$.
\end{proof}

\begin{proposition}
\label{prop:kernel_linf}
$K$ is an $\mathscr{L}_{\infty,\mu}$-space for every $\mu > \mu_0$.
\end{proposition}

\begin{proof}
By Lemma~\ref{lem:Ks_structure}, $(K_s)_s$ is an upward-directed family of
finite-dimensional subspaces of $K$ with dense union, each at
Banach--Mazur distance at most $\mu_0$ from a finite-dimensional
$\ell_\infty$. Apply Lemma~\ref{lem:directed} with $Z = K$.
\end{proof}

\begin{proposition}
\label{prop:kernel_density}
$\operatorname{dens}(K) = \kappa$.
\end{proposition}

\begin{proof}
Since $K \subseteq Y$, $\operatorname{dens}(K) \leq \kappa$. Let
$A \subseteq K$ be dense with $|A| = \operatorname{dens}(K)$, and let
$B \subseteq Y$ be a countable set with $T(B)$ dense in $c_0$
(Theorem~\ref{thm:surjectivity}). Given $y \in Y$ and $\varepsilon > 0$,
pick $b \in B$ with $\|Ty - Tb\| < \varepsilon'$, where
$\varepsilon' > 0$ is to be chosen; by Theorem~\ref{thm:surjectivity}
there is $z \in Y$ with $Tz = T(y-b)$ and
$\|z\| \leq 2\varepsilon'/(1-\eta)$, so $y - b - z \in K$. Choose
$a \in A$ with $\|(y-b-z) - a\| < \varepsilon/2$. Then
$\|y - (a+b)\| < \varepsilon/2 + 2\varepsilon'/(1-\eta) < \varepsilon$ for
small $\varepsilon'$. So $A + B$ is dense in $Y$ and
$\kappa \leq |A|\cdot\aleph_0$. If $A$ were countable, $Y$ would be
separable; hence $|A| \geq \aleph_1$ and $|A|\cdot\aleph_0 = |A|$.
\end{proof}

\begin{remark}
The exact sequence $0 \to K \to Y \to c_0 \to 0$ thus has an
$\mathscr{L}_\infty$ kernel of density $\kappa$; the
$\mathscr{L}_\infty$-constant of $K$ is controlled by
$\lambda\max(1, \eta/(1-\eta))$, compared with $\lambda$ for $Y$.
\end{remark}

\section{Remarks on Other Realizations}
\label{sec:necessity}

The reserved coordinates of Lemma~\ref{lem:realization} are what make the
functionals of Section~\ref{sec:coords} well defined. Without them a naive
coordinate functional can fail to be well defined, as the following
elementary observation shows.

\begin{proposition}
\label{prop:necessity_welldefined}
Let $S_t \subseteq \ell_\infty^{n_t}$ and $1 \leq j \leq n_t$. The map
$[(b,e)] \mapsto b_j$ is well defined on $E_t^{(s)}$ if and only if
$\xi_j = 0$ for every $\xi \in S_t$; this is strictly stronger than
$e_j \notin S_t$.
\end{proposition}

\begin{proof}
The map is well defined if and only if it vanishes on $N_t^{(s)}$, i.e.
$\xi_j = 0$ for all $\xi \in S_t$. For the last assertion take $n_t = 2$
and $S_t = \operatorname{span}\{e_1+e_2\}$: then $e_1 \notin S_t$, but
$(e_1+e_2)_1 = 1$.
\end{proof}

\begin{remark}
\label{rmk:naive}
Version~1 obtained surjectivity from an $\ell_\infty$-upper estimate for
vectors introduced at singleton stages. Such vectors enter later stages
through the second component of the pushouts~\eqref{eq:pushout}, where
$[(0,v^{(s)}_t\xi)] = [(\xi,0)]$; their $\ell_\infty$-coordinates in later
$G$-spaces are therefore images under the maps $(v^{(s)}_t)^{-1}$, which
may have norm up to $\lambda > 1$ and need not preserve disjointness of
supports. This is why the proof of Theorem~\ref{thm:surjectivity} instead
lifts finitely supported vectors at a single stage $\sigma$, using the
reserved coordinates $c_{\sigma,\gamma}$.
\end{remark}

\begin{remark}
\label{rmk:agm_comparison}
In the separable setting, Argyros, Gasparis and Motakis \cite{AM16} obtain
a $c_0$ quotient with $\mathscr{L}_\infty$ kernel by first representing
every separable $\mathscr{L}_\infty$-space as a Bourgain--Delbaen space
\cite[Theorem~3.6]{AM16}. In the realization of
Section~\ref{sec:coords_realization}, the reserved coordinates and the
recursively defined functionals of Section~\ref{sec:coords} play an
analogous role for the nonseparable Lopez-Abad spaces: they provide
coordinate functionals that are compatible with the whole directed system,
without any representation theorem. Note that existence of a $c_0$
quotient with $\mathscr{L}_\infty$ kernel, for every non-Grothendieck
$\mathscr{L}_\infty$-space, also follows from
Theorem~\ref{thm:goodkernel}; the point of Theorem~\ref{thm:explicit} is
the explicit description of the map and its kernel.
\end{remark}

\section{Conclusion and Open Questions}

For every Banach space $X$ of uncountable density, the reserved-coordinate
realization of Lopez-Abad's construction yields an $\mathscr{L}_\infty$-space
$Y \supseteq X$, with $Y/X$ Schur and Radon--Nikod\'ym, together with an
explicit surjection $T: Y \to c_0$ whose kernel is an
$\mathscr{L}_\infty$-space of the same density
(Theorem~\ref{thm:explicit}). For arbitrary realizations, every Lopez-Abad
space has an $\ell_2$ quotient, and it has a $c_0$ quotient as soon as it is
not Grothendieck (Theorem~\ref{thm:general}, Corollary~\ref{cor:lopezabad});
in that case a $c_0$ quotient with $\mathscr{L}_\infty$ kernel also exists
(Theorem~\ref{thm:goodkernel}). We close with some questions.

\begin{enumerate}
\item \textbf{Schur and Grothendieck.} Can an infinite-dimensional Banach
space be both Schur and Grothendieck (Question~\ref{q:schurgroth})? A
negative answer would follow from a positive answer to the problem of
whether every Grothendieck space has property~(V) \cite[Problem~4]{GK21}.
\item \textbf{Grothendieck Lopez-Abad spaces.} Is some realization of the
Lopez-Abad construction a Grothendieck space? By
Lemma~\ref{lem:quotientgroth}, if $Y$ were Grothendieck and $Y/X$
infinite-dimensional, then the infinite-dimensional Schur space $Y/X$ would
be Grothendieck; so a negative answer to
Question~\ref{q:schurgroth} answers this question negatively, while a
Grothendieck realization would give a positive answer to
Question~\ref{q:schurgroth}.
\item \textbf{Isomorphism type of realizations.} Different choices of
$(n_t, S_t, \upsilon_t)$ may yield non-isomorphic spaces. Are all
realizations with the same parameters $(X, D, \lambda, \eta)$ isomorphic?
\item \textbf{Explicit quotients for other realizations.} For a
realization that is not Grothendieck, can a $c_0$ quotient with
$\mathscr{L}_\infty$ kernel be described explicitly in terms of the
construction, as in Theorem~\ref{thm:explicit}?
\end{enumerate}

\appendix

\section{Kernels of Quotient Maps onto $c_0$}
\label{app:kernels}

Let $Z$ be a Banach space and $S: Z \to c_0$ a bounded operator. Writing
$(e_n)$ for the unit vector basis of $\ell_1 = c_0^*$ and
$z_n^* := S^*e_n$, we have $Sz = (z_n^*(z))_n$, and $(z_n^*)$ is
weak$^*$ null. By the open mapping and closed range theorems, $S$ is onto
if and only if $S^*$ is bounded below, i.e.\ if and only if $(z_n^*)$ is
equivalent to the unit vector basis of $\ell_1$; in that case
$S^*(\ell_1) = [z_n^*] = (\ker S)^\perp$. Conversely, every weak$^*$ null
sequence in $Z^*$ equivalent to the $\ell_1$-basis defines in this way a
surjection of $Z$ onto $c_0$.

We use the following standard facts.
\begin{enumerate}
    \item[(F1)] A Banach space $W$ is an $\mathscr{L}_\infty$-space if and
    only if $W^{**}$ is injective \cite[Proposition~4.3.3]{GK21}.
    \item[(F2)] For every Banach space $W$, $i_W^* \circ i_{W^*} =
    \mathrm{Id}_{W^*}$, where $i_W: W \to W^{**}$ is the canonical
    embedding.
    \item[(F3)] A complemented subspace of an injective space is injective.
\end{enumerate}
(F2) is immediate: $(i_W^* i_{W^*} w^*)(w) = (i_{W^*}w^*)(i_W w) =
(i_W w)(w^*) = w^*(w)$. For (F3), if $P: E \to F$ is a projection and
$E$ is injective, an operator into $F$ extends into $E$, and composing
with $P$ gives an extension into $F$.

\subsection{A criterion}

\begin{lemma}
\label{lem:kernelcriterion}
Let $S: Z \to c_0$ be a surjection, $K = \ker S$, and $z_n^* = S^*e_n$.
\begin{enumerate}
    \item If $K$ is an $\mathscr{L}_\infty$-space, then $[z_n^*]$ is
    complemented in $Z^*$.
    \item If $Z$ is an $\mathscr{L}_\infty$-space and $[z_n^*]$ is
    complemented in $Z^*$, then $K$ is an $\mathscr{L}_\infty$-space.
\end{enumerate}
\end{lemma}

\begin{proof}
Let $J: K \to Z$ be the inclusion. Its adjoint $J^*: Z^* \to K^*$ is
restriction, which is onto by Hahn--Banach and has kernel
$K^\perp = [z_n^*]$. Consequently $J^{**}$ is an isomorphic embedding.

(1) By (F1), $K^{**}$ is injective, so there is $P: Z^{**} \to K^{**}$
with $PJ^{**} = \mathrm{Id}_{K^{**}}$. Put $R := P \circ i_Z : Z \to K^{**}$.
Since $i_Z J = J^{**} i_K$, we get $RJ = PJ^{**}i_K = i_K$. Define
$Q := R^* \circ i_{K^*} : K^* \to Z^*$. Then, by (F2),
\[
J^*Q = J^*R^*i_{K^*} = (RJ)^* i_{K^*} = i_K^* i_{K^*} = \mathrm{Id}_{K^*}.
\]
Hence $I - QJ^*$ is a bounded projection on $Z^*$ with range
$\ker J^* = [z_n^*]$.

(2) Let $Z^* = [z_n^*] \oplus M$. Since $\ker J^* = [z_n^*]$, the
restriction $J^*|_M : M \to K^*$ is an isomorphism; let
$Q := (J^*|_M)^{-1}$, so $J^*Q = \mathrm{Id}_{K^*}$. Taking adjoints,
$Q^*J^{**} = \mathrm{Id}_{K^{**}}$, so $J^{**}Q^*$ is a projection of
$Z^{**}$ onto $J^{**}(K^{**}) \cong K^{**}$. By (F1), $Z^{**}$ is
injective; by (F3), $K^{**}$ is injective; by (F1) again, $K$ is an
$\mathscr{L}_\infty$-space.
\end{proof}

\subsection{A surjection with a bad kernel}

We use Bourgain's finite-dimensional examples \cite[\S2]{Bo81}. For an
integer $d \geq 4$ put $N = d^{6d}$ and $\varepsilon = 1/(4d)$. Bourgain
works in the finite atomic $L_1$-space
\[
L_1(\Omega_d) := \ell_1(d) \oplus_1 L_1(G_N^{\,d}) \oplus_1
\mathfrak{X}_d, \qquad
\mathfrak{X}_d := L_1(G_N) \oplus_1 \cdots \oplus_1 L_1(G_N)
\ (d \text{ copies}),
\]
where $G_N = \{1,-1\}^N$ carries its Haar measure; we write $w_a > 0$ for
the weight of an atom $a \in \Omega_d$. He defines an operator
$\Lambda_d : \mathfrak{X}_d \to L_1(\Omega_d)$ (his $\Lambda_\varepsilon$)
and shows:
\begin{enumerate}
    \item[(B1)] $\frac{1}{24}\|\varphi\| \leq \|\Lambda_d\varphi\| \leq
    5\|\varphi\|$ for $\varphi \in \mathfrak{X}_d$ \cite[Lemma~4]{Bo81};
    \item[(B2)] every projection of $L_1(\Omega_d)$ onto
    $R_d := \Lambda_d(\mathfrak{X}_d)$ has norm at least $d/384$
    \cite[Lemma~6]{Bo81}.
\end{enumerate}
Moreover, the functions $e^{(i)}_\nu$ ($1 \leq i \leq d$, $\nu \in G_N$),
where $e_\nu = \prod_{n=1}^N(1 + \nu_n r_n)$ placed in the $i$-th copy of
$L_1(G_N)$, are $2^N$ times the indicator functions of the points of
$G_N$, hence disjointly supported of norm one; they form an isometric
$\ell_1$-basis of $\mathfrak{X}_d$ \cite[p.~135]{Bo81}.

\begin{proposition}
\label{prop:badkernel}
There is a surjection $S: C[0,1] \to c_0$ whose kernel is not an
$\mathscr{L}_\infty$-space.
\end{proposition}

\begin{proof}
\emph{Placing the blocks.} Let $I_d = (\frac{1}{d+1}, \frac1d]$ for
$d \geq 4$, and split $I_d$ into its left and right halves $I_d^\pm$.
Partition $I_d^\pm$ into intervals $S_a^\pm$ ($a \in \Omega_d$) of
positive lengths $\ell_a^\pm$, and identify $L_1(I_d)$ with the measures
in $M[0,1] = C[0,1]^*$ that are absolutely continuous and supported in
$I_d$. Define
\[
\iota^\pm f = \sum_a f(a)\frac{w_a}{\ell_a^\pm}\mathbf{1}_{S_a^\pm}, \qquad
\kappa^\pm g\,(a) = \frac{1}{w_a}\int_{S_a^\pm} g ,
\]
and $J_d f = \frac12(\iota^+f - \iota^-f)$,
$L_d g = \kappa^+(g|_{I_d^+}) - \kappa^-(g|_{I_d^-})$. Then $\iota^\pm$
are isometries, $\|\kappa^\pm g\| \leq \|g\|$, $\kappa^\pm\iota^\pm =
\mathrm{Id}$, and consequently $J_d: L_1(\Omega_d) \to L_1(I_d)$ is an
isometry, $\|L_d\| \leq 1$, and $L_dJ_d = \mathrm{Id}$. Moreover
$\int \iota^+ f = \int \iota^- f = \sum_a f(a)w_a$, so every measure
$J_df$ has total mass $0$.

\emph{The sequence.} Let $(\mu_n)$ enumerate the measures
$J_d\Lambda_d(e^{(i)}_\nu)$, block by block for $d = 4, 5, \ldots$ (each
block is finite), and let $E := [\mu_n] \subseteq M[0,1]$.

\emph{(a) $(\mu_n)$ is weak$^*$ null.} Let $g \in C[0,1]$ and let $\mu$ be
in block $d$, with $t_d \in I_d$. As $\mu$ has mass $0$, $\|\mu\| \leq 5$
by (B1), and $\mu$ is supported in $I_d$,
\[
\Bigl|\int g\,d\mu\Bigr| = \Bigl|\int (g - g(t_d))\,d\mu\Bigr| \leq
5\sup_{t \in I_d}|g(t) - g(t_d)| ,
\]
which tends to $0$ as $d \to \infty$ by uniform continuity.

\emph{(b) $(\mu_n)$ is equivalent to the $\ell_1$-basis.} A finite linear
combination of the $\mu_n$ has the form $\sum_d J_d\Lambda_d\varphi_d$
with $\varphi_d \in \mathfrak{X}_d$, and the $\ell_1$-norm of its
coefficients is $\sum_d\|\varphi_d\|$. The summands are supported in the
disjoint sets $I_d$ and $J_d$ is an isometry, so by (B1)
\[
\frac{1}{24}\sum_d\|\varphi_d\| \leq
\Bigl\|\sum_d J_d\Lambda_d\varphi_d\Bigr\| =
\sum_d\|\Lambda_d\varphi_d\| \leq 5\sum_d\|\varphi_d\| .
\]

By (a), (b) and the discussion at the start of this appendix,
$Sg := (\int g\,d\mu_n)_n$ defines a surjection $S: C[0,1] \to c_0$ with
$S^*(\ell_1) = E$. Since $C[0,1]$ is an $\mathscr{L}_\infty$-space,
Lemma~\ref{lem:kernelcriterion}(1) shows that it suffices to prove:

\emph{(c) $E$ is not complemented in $M[0,1]$.} Suppose $P: M[0,1] \to E$
is a bounded projection. Let $\pi_d(\nu) = \nu|_{I_d}$, a contractive
projection on $M[0,1]$. Every finite combination of the $\mu_n$ is a sum
of pieces in $J_d(R_d)$ supported in the disjoint $I_d$, so
$\pi_d(E) \subseteq J_d(R_d)$ (the latter is finite-dimensional, hence
closed), and $\pi_d$ is the identity on $J_d(R_d) \subseteq E$. Define
\[
P_d := L_d \circ \pi_d \circ P \circ J_d : L_1(\Omega_d) \to
L_1(\Omega_d).
\]
Its range lies in $L_d(J_d(R_d)) = R_d$; for $r \in R_d$ we have
$J_dr \in E$, so $PJ_dr = J_dr$, $\pi_dJ_dr = J_dr$ and $P_dr = r$. Thus
$P_d$ is a projection of $L_1(\Omega_d)$ onto $R_d$ with
$\|P_d\| \leq \|P\|$. By (B2), $\|P\| \geq d/384$ for every $d \geq 4$, a
contradiction.
\end{proof}

\subsection{Some surjection always has a good kernel}

\begin{lemma}[Kadec--Pe\l czy\'nski]
\label{lem:KP}
Let $\nu$ be a probability measure and let $W \subseteq L_1(\nu)$ be a
closed subspace isomorphic to $\ell_1$. Then for every $\varepsilon > 0$
there is $x \in W$ with $\|x\|_1 = 1$ and
$\nu(|x| \geq \varepsilon) < \varepsilon$.
\end{lemma}

\begin{proof}
This is a form of the Kadec--Pe\l czy\'nski dichotomy \cite{KP62}; we give
a self-contained argument for the case we need. Suppose that for some
$\varepsilon > 0$ every $x \in W$ satisfies
$\nu(|x| \geq \varepsilon\|x\|_1) \geq \varepsilon$.

\emph{Convergence in measure implies norm convergence on $W$.} If
$w_k \in W$, $w_k \to 0$ in measure, but $\|w_k\|_1 \geq c > 0$ along a
subsequence, then along it
$\nu(|w_k| \geq \varepsilon c) \geq \nu(|w_k| \geq \varepsilon\|w_k\|_1)
\geq \varepsilon$, which contradicts $w_k \to 0$ in measure.

\emph{A contradiction.} Let $(g_n)$ be a bounded sequence in $W$
equivalent to the unit vector basis of $\ell_1$, with
$\|\sum a_ng_n\|_1 \geq c_0\sum|a_n|$ for some $c_0 > 0$. By Koml\'os'
theorem \cite{Ko67} there are a subsequence $(g_{n_i})$ and a measurable
$h$ such that the Ces\`aro means $c_m := \frac1m\sum_{i \leq m}g_{n_i}$
converge to $h$ almost everywhere. Then
$w_k := c_{2^{k+1}} - c_{2^k} \to 0$ almost everywhere, hence in measure
since $\nu$ is finite. But
\[
w_k = \frac{1}{2^{k+1}}\Bigl(\sum_{2^k < i \leq 2^{k+1}} g_{n_i}
- \sum_{i \leq 2^k} g_{n_i}\Bigr),
\]
whose coefficients have absolute values summing to $1$; hence
$\|w_k\|_1 \geq c_0$ for every $k$. This contradicts the previous
paragraph.
\end{proof}

\begin{theorem}
\label{thm:goodkernel}
Let $Y$ be an $\mathscr{L}_\infty$-space that is not Grothendieck. Then
there is a surjection $S: Y \to c_0$ whose kernel is an
$\mathscr{L}_\infty$-space.
\end{theorem}

\begin{proof}
\emph{Step 1: $Y^*$ inside an $L_1$-space.} By the proof of
Theorem~\ref{thm:linftygrothendieck}, there is an isomorphic embedding
$\Phi: Y^* \to L_1(\mu)$ for some measure $\mu$ (with complemented range),
and $Y^*$ is weakly sequentially complete.

\emph{Step 2: a weak$^*$ null $\ell_1$-sequence.} As $Y$ is not
Grothendieck, there is a weak$^*$ convergent sequence in $Y^*$ that is not
weakly convergent; subtracting its weak$^*$ limit gives a weak$^*$ null
sequence $(y_n^*)$ that is not weakly null. It is bounded. Passing to a
subsequence, there are $\Psi \in Y^{**}$ and $\delta > 0$ with
$|\Psi(y_n^*)| \geq \delta$ for all $n$. If some subsequence were weakly
Cauchy, it would converge weakly by weak sequential completeness, and its
weak limit would equal its weak$^*$ limit $0$, contradicting
$|\Psi(y_n^*)| \geq \delta$. By Rosenthal's $\ell_1$ theorem \cite{Ro74},
after passing to a further subsequence, $(y_n^*)$ is equivalent to the unit
vector basis of $\ell_1$: there is $c > 0$ with
$\|\sum a_ny_n^*\| \geq c\sum|a_n|$. Moreover $(y_n^*)$ is still weak$^*$
null. (Compare \cite[Proposition~3.1.6]{GK21}.)

\emph{Step 3: reduction to a probability measure.} Put
$f_n := \Phi y_n^*$. Each $f_n$ vanishes outside a $\sigma$-finite set, so
there is a $\sigma$-finite set $A$ outside which all $f_n$ vanish. Choose
$\omega > 0$ on $A$ with $\int_A \omega\,d\mu = 1$ and let
$\nu = \omega\mu|_A$. Then $f \mapsto f/\omega$ is an isometry of
$L_1(\mu|_A)$ onto $L_1(\nu)$ that preserves supports and satisfies
$\int_E|f|\,d\mu = \int_E |f/\omega|\,d\nu$ for every $E \subseteq A$. We
carry out Steps 4--5 in $L_1(\mu|_A)$, measuring the size of sets by
$\nu$. The restriction $\rho_A : L_1(\mu) \to L_1(\mu|_A)$ is a
contractive projection.

\emph{Step 4: concentrated blocks in every tail.} Fix $m \in \mathbb{N}$
and $\varepsilon \in (0, \frac12)$. The tail span
$W_m := [f_n : n \geq m]$ is isomorphic to $\ell_1$.
By Lemma~\ref{lem:KP} (transported to $L_1(\mu|_A)$) there is
$x \in W_m$ with $\|x\| = 1$ and, for $B := \{|x/\omega| \geq
\varepsilon/4\}$, $\nu(B) < \varepsilon/4$; then
$\int_{B^c}|x|\,d\mu = \int_{B^c}|x/\omega|\,d\nu \leq \varepsilon/4$.
Approximating $x$ within $\varepsilon/8$ by a finite linear combination of
the $f_n$ with $n \geq m$ and normalizing, we obtain:
\begin{quote}
for every $m$ and every $\varepsilon \in (0,\frac12)$ there are a finite set
$F \subseteq [m, \infty)$, scalars $(a_n)_{n \in F}$, and a set $B$ with
$\nu(B) < \varepsilon$, such that $x := \sum_{n \in F} a_nf_n$ satisfies
$\|x\| = 1$ and $\int_{B^c}|x|\,d\mu < \varepsilon$.
\end{quote}

\emph{Step 5: almost disjoint blocks.} Fix $\varepsilon_k > 0$ with
$\sum_k\varepsilon_k \leq \frac14$, and put
$\tau_k := \sum_{j \geq k}\varepsilon_j$. We choose recursively finite
sets $F_1 < F_2 < \cdots$ (i.e.\ $\max F_k < \min F_{k+1}$), normalized
$x_k = \sum_{n \in F_k}a_nf_n$, and sets $B_k \subseteq A$. Given
$x_1, \ldots, x_{k-1}$, choose $\eta_k > 0$ such that $\nu(E) < \eta_k$
implies $\int_E|x_j|\,d\mu < \varepsilon_k$ for all $j < k$ (absolute
continuity, applied to finitely many functions); then apply Step 4 with
$m > \max F_{k-1}$ and $\min(\varepsilon_k, \eta_k)$ in place of
$\varepsilon$. Thus $\nu(B_k) < \eta_k$ and
$\int_{B_k^c}|x_k|\,d\mu < \varepsilon_k$.

Let $D_k := B_k \setminus \bigcup_{j > k}B_j$; these sets are pairwise
disjoint. For $j > k$ we have $\nu(B_j) < \eta_j$, hence
$\int_{B_j}|x_k|\,d\mu < \varepsilon_j$, and so
\[
\int_{D_k^c}|x_k|\,d\mu \leq \int_{B_k^c}|x_k|\,d\mu +
\sum_{j > k}\int_{B_j}|x_k|\,d\mu < \tau_k .
\]
Let $d_k := x_k\mathbf{1}_{D_k}$ and $u_k := d_k/\|d_k\|$. Then
$\|x_k - d_k\| < \tau_k$, $\|d_k\| > 1 - \tau_k$, and
$\|x_k - u_k\| \leq \|x_k - d_k\| + (1 - \|d_k\|) < 2\tau_k \leq
2\tau_1 \leq \frac12$.

\emph{Step 6: complementation.} The $u_k$ are disjointly supported and
normalized. Let $\psi_k := \overline{\operatorname{sgn} u_k}\,
\mathbf{1}_{D_k} \in L_\infty(\mu|_A)$, so that $\langle u_j, \psi_k\rangle
= \delta_{jk}$, and $\sum_k|\langle f, \psi_k\rangle| \leq \|f\|$ for
$f \in L_1(\mu|_A)$. Hence $Q f := \sum_k\langle f,\psi_k\rangle u_k$ is a
contractive projection of $L_1(\mu|_A)$ onto $[u_k]$. Define
$Kf := \sum_k\langle f, \psi_k\rangle(x_k - u_k)$; then
$\|K\| \leq \sup_k\|x_k - u_k\| \leq \frac12$, so $U := I + K$ is an
automorphism of $L_1(\mu|_A)$, and $Uu_k = x_k$. Therefore
$P_1 := UQU^{-1}\rho_A$ is a bounded projection of $L_1(\mu)$ onto
$[x_k] = U([u_k])$, and $(x_k)$ is equivalent to the $\ell_1$-basis.

Let $z_k^* := \Phi^{-1}x_k = \sum_{n \in F_k}a_ny_n^*$. Since
$[x_k] \subseteq \Phi(Y^*)$, the operator $\Phi^{-1}P_1\Phi$ is a bounded
projection of $Y^*$ onto $[z_k^*]$, and $(z_k^*)$ is equivalent to the
$\ell_1$-basis.

\emph{Step 7: weak$^*$ nullity and conclusion.} By Step 2,
$\sum_{n \in F_k}|a_n| \leq c^{-1}\|z_k^*\| \leq c^{-1}\|\Phi^{-1}\|$.
Hence for $y \in Y$,
\[
|z_k^*(y)| \leq c^{-1}\|\Phi^{-1}\|\,\sup_{n \geq \min F_k}|y_n^*(y)|
\longrightarrow 0,
\]
since $\min F_k \to \infty$ and $(y_n^*)$ is weak$^*$ null. Thus
$Sy := (z_k^*(y))_k$ is a surjection of $Y$ onto $c_0$ with
$S^*(\ell_1) = [z_k^*]$ complemented in $Y^*$, and its kernel is an
$\mathscr{L}_\infty$-space by Lemma~\ref{lem:kernelcriterion}(2).
\end{proof}

\begin{remark}
In the separable case, Theorem~\ref{thm:goodkernel} is the corollary of
Argyros--Gasparis--Motakis \cite{AM16} mentioned in the introduction; the
argument above avoids their Bourgain--Delbaen representation.
\end{remark}

\section*{Acknowledgments}

The author thanks Dr.~Robert Pluta (Northeastern University) and
Dr.~Pavlos Motakis (York University) for invaluable guidance, detailed
feedback, and suggestions that shaped this work, and the anonymous
referees of an earlier version for their comments.

\end{document}